\documentclass[letter, 11pt]{article}

\usepackage{natbib}
\usepackage{amssymb,amsmath,xcolor,graphicx,xspace,colortbl,rotating} % 
\usepackage[raggedrightboxes]{ragged2e} % ,revsymb4-1}
\usepackage{textcomp}
\usepackage{amssymb,amsmath,xcolor,graphicx,xspace,colortbl,rotating}  %% 100
\usepackage{appendix}  %% 100
\usepackage{bm}  %% 100
\usepackage{boxedminipage}  %% 100
\usepackage{color}  %% 100
\usepackage{endnotes}  %% 100
\usepackage{pgf,tikz,pgfplots}
\pgfplotsset{width=5cm,compat=1.16}
\usepackage{mathrsfs}
\usetikzlibrary{arrows}
\usepackage{caption,subcaption}
\usepackage[margin=0.9in]{geometry}
\usepackage[raggedrightboxes]{ragged2e}  %% 100
\usepackage[onehalfspacing]{setspace}  %% 100
\usepackage{tabulary}  %% 100
\usepackage{textcomp}  %% 100
\usepackage{varioref}  %% 100
\usepackage{wrapfig}  %% 100
\graphicspath{{../graphics/}{../tcache/}{../gcache/}}
\DeclareGraphicsExtensions{.pdf,.eps,.ps,.png,.jpg,.jpeg}

\newtheorem {theorem}{Theorem}

\newtheorem {corollary}{Corollary}

\newtheorem {example}{Example}

\newtheorem {lemma}{Lemma}

\newtheorem {proposition}{Proposition}
\newtheorem {remark}{Remark}

\newenvironment {proof}[1][Proof]{\noindent \textbf {#1.} }{\ \rule {0.5em}{0.5em}}

\begin{document}

\setlength{\abovedisplayskip}{3pt}
\setlength{\belowdisplayskip}{3pt}

\title{Concentration inequalities using higher moments information}
\author{Bar Light\protect\footnote{ 
Microsoft Research, NY, USA. e-mail: \textsf{barlight@microsoft.com}\ }
~ ~
}
\maketitle

\thispagestyle{empty}

\noindent \noindent \textsc{Abstract}:

  In this paper, we generalize and improve some fundamental concentration inequalities using information on the random variables' higher moments. In particular, we improve the classical Hoeffding's and Bennett's inequalities for the case where there is some information on the random variables' first $p$ moments for every positive integer $p$. Importantly, our generalized Hoeffding's inequality is tighter than Hoeffding's inequality and is given in a simple closed-form expression for every positive integer $p$. Hence, the generalized Hoeffding's inequality is easy to use in applications. To prove our results, we derive novel upper bounds on the moment-generating function of a random variable that depend on the random variable's first $p$ moments and show that these bounds satisfy appropriate convexity properties.

\relax
%\smallskip \noindent \emph{} 
\newpage

\section{Introduction}

Concentration inequalities provide bounds on the probability that a random variable differs from some value, typically the random variable's expected value  (see \cite{boucheron2013concentration} for a textbook treatment of concentration inequalities). Besides their importance in probability theory,  concentration inequalities are an important mathematical tool in statistics and operations research (see \cite{massart2000some}), the analysis of algorithms and machine learning theory (see \cite{alon2004probabilistic} and \cite{mohri2018foundations}) and many other fields.  Two of the most important and useful concentration inequalities are Hoeffding's inequality \citep{hoeffding1994probability} and Bennett's inequality \citep{bennett1962probability}. These are inequalities that bound the probability that the sum of independent random variables differs from its expected value. The bound derived in Hoeffding's inequality holds for bounded random variables and uses information on the random variables' first moment. The bound derived in Bennett's inequality holds for random variables that are bounded from above and uses information on the random variables' first and second moments. Despite their importance and numerous generalizations\footnote{There are many extensions and generalizations of Hoeffding's and Bennett's inequalities. For example see \cite{freedman1975tail},  \cite{pinelis1994optimum},  \cite{talagrand1995missing}, \cite{roussas1996exponential},  \cite{cohen1999optimal}, \cite{victor1999general}, \cite{bousquet2002bennett}, \cite{bentkus2004hoeffding}, \cite{klein2005concentration}, \cite{kontorovich2008concentration}, \cite{fan2012hoeffding},  \cite{junge2013noncommutative},   \cite{pinelis2014bennett}, \cite{paulin2015concentration}, \cite{pelekis2015bernstein}, \cite{jiang2018bernstein}, and \cite{pepin2021concentration}. }, there are not many 
improvements even for the basic case of sums of independent real-valued random variables \citep{pinelis2014bennett}, especially concentration bounds that use information on higher order moments and are given as a simple closed-form expression. 

In this paper we generalize and improve Bennett's and Hoeffding's inequalities. We provide bounds that use information on the random variables' higher moments. More precisely, we provide bounds on the probability that the sum of independent random variables differs from its expected value where the bounds depend on the random variables' first $p$ moments for every integer $p \geq 1$. We provide two families of concentration inequalities, one that generalizes Hoeffding's inequality and one that generalizes Bennett's inequality. Importantly, the bounds that we derive are tighter than Bennett's and Hoeffding's inequalities and are given as closed-form expressions in most cases.  In our generalized Hoeffding's inequality, our bounds hold for bounded random variables and are given as simple closed-form expressions (see Theorem \ref{Thm: concent hof}) for every integer $p \geq 1$. In our generalized Bennett's inequality, our bounds hold for random variables that are bounded from above. For $p=3$, our bound is given in a closed-form expression in terms of the Lambert $W$-function. This bound uses information on the random variables' first three moments and is tighter than Bennett's inequality. For $p > 3$ our bounds are given in terms of the generalized Lambert $W$-function (see Theorem \ref{Thm: concent bennett}). 

%For any integer $p \geq 1$ the bound uses information on the random variables' first $p$ moments and is always tighter than Hoeffding's inequality. These bounds can be generalized for martingales and other stochastic processes in a standard way. In our generalized Bennett's inequality, our bounds hold for random variables that are bounded from above. For $p=3$, our bound is given in a closed-form expression in terms of the Lambert $W$-function. This bound uses information on the random variables' first three moments and is tighter than Bennett's inequality. For $p > 3$ our bounds are given in terms of the generalized Lambert $W$-function (see Theorem \ref{Thm: concent bennett}). 

For every positive integer $p$, independent random variables $X_{1},\ldots,X_{n}$ such that $\mathbb{P}(X_{i} \in [a_{i},b_{i}] = 1)$, and all $t>0$,  our generalized Hoeffding's inequality is given by 
\begin{equation*}
    \mathbb{P}(S_{n}-\mathbb{E}(S_{n}) \geq t) 
      \leq \exp \left (-\frac{2t^{2}}{\sum _{i=1}^{n} (b_{i} - a_{i})^{2}  C_{p}(t,X_{i}) } \right )
\end{equation*}
where $S_{n}= \sum _{i=1}^{n} X_{i}$ and $C_{p}(t,X_{i})$ is a function that depends on $t$, on the first $p$ moments of $X_{i}$,  and on $X_{i}$'s support: $[a_{i},b_{i}]$. We show that for every positive integer $p$ we have $C_{p} \leq 1$. Thus,  our generalized Hoeffding's inequality is tighter than Hoeffding's inequality which corresponds to $p=1$ and $C_{1}=1$. We provide a simple closed-form expression for the function $C_{p}$ for any integer $p \geq 1$. For example, suppose that the support of a random variable $X$ is $[0,b]$ for some $X=X_{i}$, $i=1,\ldots,n$. Then $C_{p}(t,X)$ is given by 
\begin{equation*}
    C_{p}(t,X) = \left ( \frac{\mathbb{E}(X^{p}) \exp(y) + \sum _{j=0} ^{p-3} \frac{y^{j}} {j!} \left (  b^{p-j-2}\mathbb{E}(X^{j+2}) - \mathbb{E}(X^{p}) \right )}{\mathbb{E}(X^{p})\exp(y) + \sum _{j=0} ^{p-2} \frac{y^{j}} {j!} \left (  b^{p-j-1}\mathbb{E}(X^{j+1}) - \mathbb{E}(X^{p}) \right )} \right )^{2}
\end{equation*}
where $y = 4tb/ \sum_{i=1}^{n} d (X_{i})$ and $d(X_{i}) = \left (\mathbb{E}(X_{i} ^{2} )/ \mathbb{E}(X_{i}) \right )^{2} $
(see Theorem \ref{Thm: concent hof}). 
We note that our generalized Hoeffding's bounds are exponential bounds, and hence, these bounds are not optimal in the sense that there is a missing factor in those bounds (see \cite{talagrand1995missing}). However, we show that the results in \cite{talagrand1995missing} can be easily adapted to our setting to obtain a concentration bound of optimal order that uses information about the random variables' higher moments. In addition, our bounds can be generalized for martingales and other stochastic processes in a standard way

To prove our concentration bounds we derive novel upper bounds on the random variable's moment-generating function that depend on the random variable’s first $p$ moments. These bounds  satisfy appropriate convexity properties that imply that we can derive a closed-form expression concentration bounds.

\section{Main results}

In this section we state our main results. In Section \ref{Sec: bounds on the MGF} we derive upper bounds on the moment-generating function of a random variable that is bounded from above. In Section \ref{Sec: concert Hoeffiding} we derive our generalized Hoeffding's inequalities. In Section \ref{Sec: concert Bennet} we derive our generalized Bennett's inequalities. 

We first introduce some notations.

Throughout the paper we consider a fixed probability space $\left (\Omega ,\mathcal{F} ,\mathbb{P}\right )$. A random variable $X$ is a measurable real-valued function from $\Omega $ to $\mathbb{R}$. We denote the expectation of a random variable on the probability space $\left (\Omega  ,\mathcal{F} ,\mathbb{P}\right )$ by $\mathbb{E}$. For $1 \leq p \leq \infty $ let $L^{p} : =L^{p}\left (\Omega  ,\mathcal{F} ,\mathbb{P}\right )$ be the space of all random variables $X :\Omega  \rightarrow \mathbb{R}$ such that $\left \Vert X\right \Vert _{p}$ is finite, where $\left \Vert X\right \Vert _{p} =\left (\int _{\Omega }\left \vert X(\omega )\right \vert ^{p} \mathbb{P}(d \omega) \right  )^{1/p}$ for $1 \leq p <\infty $ and $\left \Vert X\right \Vert _{p} =\operatorname{ess\,sup}\left \vert X(\omega )\right \vert $ for $p =\infty $. We say that $X$ is a random variable on $[a,b]$ for some $a<b$ if $\mathbb{P}(X \in [a,b]) = 1$.

For $k\geq 1$, we denote by $f^{(k)}$ the $k$th derivative of a $k$ times differentiable function $f :[a ,b] \rightarrow \mathbb{R}$ and for $k=0$ we define $f^{(0)}:=f$. As usual, the derivatives at the extreme points $f^{(k)}(a)$ and $f^{(k)}(b)$ are defined by taking the left-side and right-side limits, respectively. We say that $f$ is increasing if $f(x) \geq f(y)$ for all $x \geq y$. 

For the rest of the paper, for every positive integer $p$, we define  $$T_{p}(x):= \exp (x) - \sum _{j=0} ^{p-2} \frac { x^{j} } {j!}$$ 
to be the Taylor remainder of the exponential function of order $p-2$ at the point $0$.  We use the convention that $\sum _{j=0} ^{k} a_{j} =0$ whenever $k < 0$ so $T_{1}(x) = \exp(x)$. The function $T_{p}$ plays an important role in our analysis.  

\subsection{Upper bounds on the moment-generating function} \label{Sec: bounds on the MGF}

In this section we provide upper bounds on the moment-generating function of a random variable that is bounded from above. We show that
\begin{equation} \label{Ineq: Wp and Wb}
    \frac{T_{p+1}(x)}{T_{p+1}(b)} \leq  \frac{\max (x^{p},0)}{b^{p}}
\end{equation} 
for all $x \leq b$, $b > 0$ and every positive integer $p$ (see the proof of Theorem \ref{Coro: exp upper bound}). This bound on the ratio of the Taylor remainders is the key ingredient in deriving the upper bounds on the moment-generating function. The proof of Bennett's inequality uses inequality (\ref{Ineq: Wp and Wb}) with $p=2$ to bound the moment-generating function (see \cite{boucheron2013concentration}). We use inequality (\ref{Ineq: Wp and Wb}) to provide upper bounds on the moment-generating function using information on the random variable's first $p$ moments for every positive integer $p$. Section \ref{Sec: Proofs}  contains the proofs not presented in the main text. 

\begin{theorem} \label{Coro: exp upper bound}
Let $X \in L^{p-1}$ be a random variable on $(-\infty ,b]$ for some $b>0$ where $p$ is a positive integer. For all $s \geq 0$ we have
\begin{align} \label{Ineq: MGF bound}
\begin{split}
  \mathbb{E} \exp(sX) & \leq \frac{\mathbb{E} \max(X^{p},0)}{b^{p}}\left (\exp(sb) - \sum _{j=0} ^{p-1} \frac{s^{j}b^{j}}{j!}  \right )+ \mathbb{E} \left ( \sum _{j=0} ^{p-1} \frac{s^{j}X^{j}} {j!} \right )   \\
   & =  \frac{\mathbb{E} \max (X^{p},0) }{b^{p}} T_{p+1}(sb)  + \mathbb{E} \left ( \sum _{j=0} ^{p-1} \frac{s^{j}X^{j}} {j!} \right ). 
   \end{split}
\end{align}  
\end{theorem}

Theorem \ref{Coro: exp upper bound} provides a unified approach for seemingly independent bounds on the moment-generating function that were derived in previous literature and used to prove concentration inequalities. 

For $p=2$, and for a random variable $X$ on $(-\infty,b]$, Theorem \ref{Coro: exp upper bound} yields the inequality 
\begin{equation} \label{Ineq: Bound on MGF p=2}
  \mathbb{E}\exp(sX) \leq   \frac{\mathbb{E}(X^{2})}{b^{2}}\left (\exp(sb) - 1 - sb \right ) + 1 +  s\mathbb{E}(X) 
\end{equation}
which is fundamental in proving Bennett's inequality (see \cite{bennett1962probability}). For $p=3$, denoting $\mu^{3}=\mathbb{E} \max(X^{3},0)$, we have 
\begin{align*} 
\begin{split}
  \frac{\mu^{3}}{b^{3}}T_{4}(sb) +   \mathbb{E} \left ( \sum _{j=0} ^{2} \frac{s^{j}X^{j}} {j!} \right )  & =   \frac{\mu^{3}}{b^{3}}T_{3}(sb)  + 1 + s\mathbb{E}(X) + \frac{s^{2}}{2}\left (\mathbb{E}(X^{2}) -  \frac{\mu^{3}}{b} \right )  \\
   & \leq \exp \left (\frac{\mu^{3}}{b^{3}}T_{3}(sb)  + s\mathbb{E}(X) + \frac{s^{2}}{2}\left (\mathbb{E}(X^{2}) -  \frac{\mu^{3}}{b} \right ) \right ). 
   \end{split}
\end{align*}  
The last inequality follows from the elementary inequality $1  + x \leq \exp(x) $ for all $x \in \mathbb{R}$. Thus, Theorem \ref{Coro: exp upper bound} implies
$$  \mathbb{E} \exp(sX) \leq \exp \left (\frac{\mathbb{E} \max(X^{3},0)}{b^{3}}T_{3}(sb)  + s\mathbb{E}(X) + \frac{s^{2}}{2}\left (\mathbb{E}(X^{2}) -  \frac{\mathbb{E} \max(X^{3},0)}{b^{3}} \right ) \right )$$
which is proved in Theorem 2 in \cite{pinelis1990exact}. 

For a random variable $X$ on $[0,b]$ let
 $$m_{X,s}(p) := \frac{\mathbb{E} \max (X^{p},0) }{b^{p}} T_{p+1}(sb)  + \mathbb{E} \left ( \sum _{j=0} ^{p-1} \frac{s^{j}X^{j}} {j!} \right ) $$ 
be the right-hand side of inequality (\ref{Ineq: MGF bound}). The next proposition shows that for every even number $p$ and $s >0$ we have $m_{X,s}(p) \geq m_{X,s}(p+1)$. If, in addition, the random variable $X$ is non-negative, then we also have $m_{X,s}(p+1) \geq m_{X,s}(p+2)$, and hence, $m_{X,s}(p)$ is decreasing. Thus, for non-negative random variables, inequality (\ref{Ineq: MGF bound}) is tighter when $p$ increases. In particular, we have $m_{X,s}(2) \geq m_{X,s}(p)$ for every integer $p \geq 3$, i.e., the bound on the moment-generating function given in inequality (\ref{Ineq: MGF bound}) is tighter than Bennett's bound (\ref{Ineq: Bound on MGF p=2}) for every integer $p \geq 3$ when $X$ is non-negative.

\begin{proposition} \label{prop: m(p)}
 Let $X \in L^{p}$ be a random variable on $(-\infty,b]$. Let $p \geq 2$ be an even number and $s > 0$. The following statements hold: 

(i) $m_{X,s}(p) \geq m_{X,s}(p+1)$. 

(ii) If $X \geq 0$ then $m_{X,s}(p+1) \geq m_{X,s}(p+2)$. 
\end{proposition}

 Note that even for $p=1$ there exists a random variable that achieves equality in (\ref{Ineq: MGF bound}). For example, a Bernoulli random variable that yields $1$ with probability $q$ and $0$ with probability $1-q$ achieves equality in (\ref{Ineq: MGF bound}) for $p=1$. For the  Bernoulli random variable all the moments are equal to $q$ which is the highest value that the higher moments can have given that the first moment equals $q$ and the support is $[0,1]$. Thus, higher moments do not provide any useful information and for every integer $p>1$ inequality (\ref{Ineq: MGF bound}) reduces to the case of $p=1$.

  The upper bounds on the moment-generating function (\ref{Ineq: MGF bound}) are not optimal in the sense that there might be a smaller bound given the information on the random variable's first $p$ moments. The optimal bound can be found by solving a linear program and is typically not given as a closed-form expression (see \cite{pandit2006worst} for a discussion). The main advantage of our upper bounds is the fact that the derivative of the right-hand-side of inequality (\ref{Ineq: MGF bound}) with respect to $s$ is log-convex for non-negative random variables. This key convexity property is the main ingredient in deriving a closed-form Hoeffding type concentration bounds that depend on the random variables' first $p$ moments (see the discussion after Theorem \ref{Thm: concent hof}). 
  For a proof of Lemma \ref{Lemma: convexity} see the proof of Theorem \ref{Thm: concent hof}. 
  \begin{lemma}\label{Lemma: convexity}
  Let $p$ be a positive integer and suppose that $X$ is random variable on $[0,b]$. Then the derivative of $z(s)$ where $z(s)$ is the right-hand-side of inequality (\ref{Ineq: MGF bound}), 
  $$ z(s):= \frac{\mathbb{E} (X^{p}) }{b^{p}} T_{p+1}(sb)  + \mathbb{E} \left ( \sum _{j=0} ^{p-1} \frac{s^{j}X^{j}} {j!} \right )$$
  is log-convex on $(0,\infty)$, i.e., $\log (z^{(1)}(s))$ is a convex function on $(0,\infty)$.  
   \end{lemma}

\subsection{Concentration inequalities: Hoeffding type inequalities} \label{Sec: concert Hoeffiding}

In this section we derive Hoeffding type concentration inequalities that provide exponential bounds on the probability that the sum of independent bounded random variables differs from its expected value. We improve Hoeffding's inequality  by using information on the random variables' first $p$ moments. We derive a tighter bound  than the standard Hoeffding's bound for every integer $p \geq 2$ (see Theorem \ref{Thm: concent hof} part (ii)). Importantly, for every positive integer $p$, the bound is given as a simple closed-form expression that depends on the random variables' first $p$ moments.

\begin{theorem} \label{Thm: concent hof}
Let $X_{1},\ldots,X_{n}$ be independent random variables where $X_{i}$ is a random variable on $[0,b_{i}]$, $b_{i} >0$, $\mathbb{P}(X_{i} > 0) > 0$. Let $S_{n} = \sum _{i=1} ^{n} X_{i}$.  Let $p \geq 1$ be an integer. 
Denote $\mathbb{E}(X^{k}_{i}) = \mu ^{k}_{i}$ and let  $D_{n} = \sum _{j=1} ^{n} d(X_{j}) $ where $d(X_{i}) = \left (\mu^{2}_{i} / \mu ^{1}_{i}\right )^{2}$. 

(i) For all $t>0$ we have
\begin{equation} \label{Ineq: Hoef bound}
     \mathbb{P}(S_{n}-\mathbb{E}(S_{n}) \geq t) \leq \exp \left (- \frac{2t^{2}}{\sum_{i=1}^{n} b^{2} _{i} C_{p} \left (4tb_{i}/D_{n},b_{i}, \mu  ^{1}_{i}, \ldots , \mu ^{p}_{i} \right ) } \right )
\end{equation}
where 
\begin{equation}
    C_{p} \left (y  ,b_{i} , \mu^{1} _{i} , \ldots , \mu^{p}_{i} \right ) = \left ( \frac{\mu^{p}_{i}\exp(y) + \sum _{j=0} ^{p-3} \frac{y^{j}} {j!} \left (  b^{p-j-2}_{i}\mu^{j+2}_{i} - \mu^{p}_{i} \right )}{\mu^{p}_{i}\exp(y) + \sum _{j=0} ^{p-2} \frac{y^{j}} {j!} \left (  b^{p-j-1}_{i}\mu^{j+1}_{i} - \mu^{p}_{i} \right )} \right )^{2}
\end{equation}
for $i=1,\ldots,n$ and all $y > 0$. 

(ii) For every integer $p \geq 1$ we have $0 < C_{p} \leq 1$. Thus, inequality (\ref{Ineq: Hoef bound}) is tighter than Hoeffding's inequality: 
\begin{equation} \label{Ineq: Hoef bound p=1}
     \mathbb{P}(S_{n}-\mathbb{E}(S_{n}) \geq t) \leq \exp \left (-\frac{2t^{2}}{\sum _{i=1}^{n} b^{2} _{i} } \right )
\end{equation}
which corresponds to $p=1$ and $C_{1} = 1$.

\end{theorem}

\begin{remark} \label{Remark: Hoeff}
(i) Theorem \ref{Thm: concent hof} can be easily applied to bounded random variables that are not necessarily positive. If  $Y_{i}$ is a random variable on $[a_{i},b_{i}]$ and $Y_{1},\ldots,Y_{n}$ are independent, we can define the random variables  $X_{i}=Y_{i}-a_{i}$ on $[0,b_{i}-a_{i}]$ and use Theorem \ref{Thm: concent hof} to conclude that 
\begin{align*} 
\begin{split}
     \mathbb{P} \left ( \sum _{i=1}^{n} Y_{i}-\mathbb{E} \left (\sum _{i=1}^{n} Y_{i} \right ) \geq t \right) & =  \mathbb{P}(S_{n}-\mathbb{E}(S_{n}) \geq t) \\
     & \leq \exp \left (-\frac{2t^{2}}{\sum _{i=1}^{n} (b_{i} - a_{i})^{2}  C_{p} \left (4t(b_{i}-a_{i})/D_{n} ,b_{i}-a_{i} , \mu^{1} _{i} , \ldots , \mu^{p}_{i} \right )} \right ).
     \end{split}
     \end{align*}
     Note that $\mu^{k}_{i} = \mathbb{E}(Y_{i}-a_{i})^{k}$ for $i=1,\ldots,n$ and  $k=1,\ldots,p$. 
     
    Applying the last inequality to $-Y_{i}$ and using the union bound yield
    \begin{align} \label{Ineq: Two-sided hoef}
     \mathbb{P} \left ( \left | \sum _{i=1}^{n} Y_{i}-\mathbb{E} \left (\sum _{i=1}^{n} Y_{i} \right  )  \right | \geq t \right) 
     \leq 2 \exp \left (-\frac{2t^{2}}{\sum _{i=1}^{n} (b_{i} - a_{i})^{2}  \overline{C}_{p} (t,Y_{i}, b_{i} - a_{i} ) } \right )
     \end{align}
   where 
   \begin{equation} \label{Equation: Cp overline}
    \overline{C}_{p} (t,Y_{i},d_{i} ) = \max  \left \{ C_{p} \left (4td_{i}/\sum_{i=1}^{n} (\lambda_{i}^{2} / \lambda_{i})^{2}, d_{i} , \lambda^{1} _{i} , \ldots , \lambda^{p}_{i} \right ) ,  C_{p} \left (4td_{i}/\sum_{i=1}^{n} (\mu_{i}^{2} / \mu_{i})^{2} ,d_{i}, \mu^{1} _{i} , \ldots , \mu^{p}_{i} \right )  \right \}
   \end{equation}
   where $\mu^{k}_{i} = \mathbb{E}(Y_{i}-a_{i})^{k}$ and $\lambda^{k}_{i} = \mathbb{E}(b_{i} - Y_{i})^{k}$
   
     (ii) If $X_{1},\ldots,X_{n}$ are identically distributed then inequality (\ref{Ineq: Hoef bound}) yields 
\begin{equation} \label{Ineq: Hoef IID}
     \mathbb{P}(S_{n}-\mathbb{E}(S_{n}) \geq nt) \leq \exp \left (-\frac{2nt^{2}}{ b^{2} _{i} C_{p} \left (4tb_{i}/d(X_{i}) ,b_{i} , \mu^{1} _{i} , \ldots , \mu^{p}_{i} \right )} \right ).
\end{equation}

(iii) In some cases of interest only a bound on the random variables' higher moments is a available. Theorem \ref{Thm: concent hof} part (i) holds also under the condition $\mathbb{E}(X^{k}_{i}) \leq \mu ^{k}_{i}$ for $k=2,\ldots,p$, $i=1,\ldots,n$ as long as $(\mu^{d+1}_{i})^{2} \leq \mu^{d}_{i} \mu^{d+2}_{i}$ for $i= 1,\ldots,n$ and $d=1,\ldots,p-2$. 

(iv) Our results can be extended in a standard way for martingales and other stochastic processes such as Markov chain (see \cite{freedman1975tail}). For the sake of brevity we omit the details. 
\end{remark}

\textbf{The approach.} We now discuss the sketch of the proof of Theorem \ref{Thm: concent hof} part (i). The full proof is in Section \ref{Sec: Proofs}. Fix a positive integer $p$. 
We start with a random variable $X$ on $[0,b]$. Assume for simplicity that $b=1$. 
Suppose that we have $\mathbb{E}\exp(yX) \leq v(y)$ for all  $y \geq 0$ where $v(y)$ is some bound on the moment generating function of $X$. Let $g(y)=\ln(v(y))$. Then using Taylor's theorem we can proceed as in the proof of Hoeffding's inequality \citep{hoeffding1994probability} to show that  
$$\mathbb{E}\exp(sX-s\mathbb{E}(X)) \leq \exp(0.5s^{2}\max _{0 \leq y \leq s} g^{(2)}(y) ).$$ 
Let $C(s)=\max _{0 \leq y \leq s}  (v^{(2)}(y)/v^{(1)}(y))^{2}$. 
We have
\begin{align*}
    \max _{0 \leq z \leq s} g^{(2)}(y) = \max _{0 \leq y \leq s} \frac{v^{(2)}(y)}{v(y)}\left (1 - \frac{v^{(2)}(y)(v^{(1)}(y))^{2}}{v(y)(v^{(2)}(y))^{2}} \right ) \leq \max _{0 \leq y \leq s} \frac{v^{(2)}(y)}{v(y)}\left (1 - \frac{v^{(2)}(y)}{v(y)C(s)} \right ) \leq  0.25C(s)
\end{align*}
where the second inequality follows from the elementary inequality $x(1-x/z) \leq 0.25z$ for all $z>0$ and $x>0$.
Hence, an essential step in deriving a closed-form exponential bound on the moment generating function is to find a function $v$ that induces a simple closed-form expression for the function $C$. This is exactly where Theorem \ref{Coro: exp upper bound} is useful. Suppose that $v(y)$ is  the right-hand side of inequality (\ref{Ineq: MGF bound}) (for some $p$). Then $v(y)$ provides a bound on  the moment generating function that depends on the random variable's first $p$ moments. 
A key step in the proof of Theorem \ref{Thm: concent hof} is to use Lemma \ref{Lemma: convexity}, i.e., to use the fact that $v^{(1)}$ is a log-convex function. This implies that $v^{(2)}/v^{(1)}$ is increasing so  $C(s)=(v^{(2)}(s)/v^{(1)}(s))^{2}$ is given in a closed-form expression. This key step shows the usefulness of Theorem \ref{Coro: exp upper bound}  for deriving closed-form concentration bounds using higher moments information.  
With this bound we can conclude that $\mathbb{E}\exp(sX-s\mathbb{E}(X)) \leq \exp(s^{2}C(s)/8)$. Applying the Chernoff bound and choosing a specific value for $s$ proves Theorem \ref{Thm: concent hof} part (i).

 \textbf{A simple special case.} The calculation of $C_{p}$ in inequality (\ref{Ineq: Hoef bound}) is immediate. For example,  
for $p=2$ we have 
$$ C_{2}(x,b_{i},\mu^{1}_{i},\mu^{2}_{i}) = \left ( \frac{\mu^{2} _{i}\exp(x)}{\mu^{2} _{i}\exp(x) + b_{i}\mu^{1} _{i} - \mu^{2} _{i}} \right ) ^{2} $$
and for $p=3$ we have 
\begin{align*}
     C_{3}(x,b_{i},\mu^{1}_{i},\mu^{2}_{i},\mu^{3}_{i}) =  \left ( \frac{\mu^{3} _{i}\exp(x) + b_{i} \mu^{2} _{i}- \mu^{3} _{i} }{\mu^{3} _{i}\exp(x) +  b^{2} _{i} \mu^{1} _{i} - \mu^{3} _{i} + (b_{i} \mu^{2} _{i} -  \mu^{3} _{i} )x} \right ) ^{2}  
     \end{align*}
     for all $i=1,\ldots,n$ and all $x \geq 0$. 

The dependence of $C_{p}$ on the first argument in Theorem \ref{Thm: concent hof}  can be simplified. 
For a random variable $X$ on $[0,1]$, a positive integer $p$ and $c>0$ let 
$$ I_{p}(X,c) =: \frac{ T_{p}(c) \mathbb{E}(X^{p}) + \sum _{j=0}^{p-2}\mathbb{E}(X^{j+1}) c^{j}/j! }{T_{p-1}(c) \mathbb{E}(X^{p}) + \sum _{j=0}^{p-3}\mathbb{E}(X^{j+2}) c^{j}/j!} $$
where $(x)_{+}=\max(x,0)$. We have $I_{p} \geq 1$ for every positive integer $p$. The function $I_{p}$ can be interpreted as a measure for the usefulness of knowing the random variable's first $p$ moment given that we know the  first $p-1$ moments. For example, $I_{1}(X,c)=1$ for every random variable $X$ and 
$$I_{2}(X,1) =\frac{ \exp(1)\mathbb{E}(X^{2}) + \mathbb{E}(X) -\mathbb{E}(X^{2})  } {\exp(1) \mathbb{E}(X^{2} )} = \frac{\exp(1)-1}{\exp(1)}+\frac{1}{\exp(1)}\frac{\mathbb{E}(X)}{\mathbb{E}(X^{2})} $$ equals $1$ if knowing  $\mathbb{E}(X^{2})$ is not useful at all given the knowledge of $\mathbb{E}(X)$ (because $\mathbb{E}(X^{2})$ is bounded above by $\mathbb{E}(X)$, then the  highest second moment possible is  $\mathbb{E}(X)$ which yields $I_{2} =1$) and is greater than $1$ when the second moment provides useful information, i.e., the first two moments differ. Similarly, $$I_{3} (X,1)= \frac{ \exp(1) \mathbb{E}(X^{3}) +\mathbb{E}(X^{2}) + \mathbb{E}(X) -2\mathbb{E}(X^{3}) } {\exp(1) \mathbb{E}(X^{3}) + \mathbb{E}(X^{2}) - \mathbb{E}(X^{3}) }$$ 
equals $I_{2}(X,1)$ when the information about the third moment is not useful (i.e., $\mathbb{E}(X^{3}) = \mathbb{E}(X^{2})$). 

We now provide a concentration bounds that depend on $I_{p}$ and simplify the concentration bounds in Theorem \ref{Thm: concent hof} when $t$ is relatively small by using the fact that $C_{p}$ is increasing in the first argument. 

\begin{corollary} \label{Coro: Special case}
Let $X_{1},\ldots,X_{n}$ be independent random variables  where $X_{i}$ is a random variable on $[0,1]$.  Let $S_{n} = \sum _{i=1} ^{n} X_{i}$. 
Assume that $\mathbb{E}(X^{k}_{i}) = \mathbb{E}(X^{k}_{j}) = \mathbb{E} (X^{k}) >0$ for $k=1,\ldots,p$ and $i,j = 1,\ldots, n$. Let $c>0$ and suppose that $$ t \leq c \left ( \frac{\mathbb{E}(X^{2} )}{2 \mathbb{E}(X)  } \right )^{2}.  $$
Then for every positive integer $p$ we have
\begin{equation} \label{Ineq: HoefBound Cp}
     \mathbb{P}(S_{n}-\mathbb{E}(S_{n}) \geq nt) \leq \exp \left (-2nt^{2} (I_{p}(X,c) ) ^{2} \right ).
\end{equation}
\end{corollary}

\textbf{The missing factor in Hoeffding's inequality.} Consider the function $$ \theta(x) : = \frac{1}{(2\pi)^{0.5} } \exp(x^{2}/2) \int _{x}^{\infty} \exp(-u^{2}/2) du$$
Then the central limit theorem and the fact that (see \cite{talagrand1995missing})
$$ \frac{1}{(2\pi)^{0.5} (1+x) } \leq  \theta(x)  \leq \frac{1}{(2\pi)^{0.5} x} $$
for all $x >0$ imply that there is a missing factor in Hoeffding's inequality \citep{talagrand1995missing}. We use the results in \cite{talagrand1995missing} together with Theorem \ref{Thm: concent hof} to derive concentration bounds of optimal order that depend on the random variables' first $p$ moments. 

\begin{corollary} \label{Coro: Talag}
Let $X_{1},\ldots,X_{n}$ be independent random variables where $\mathbb{E}(X_{i})=0$ and $\vert X_{i} \vert \leq b$ for $i=1,\ldots,n$. Let $S_{n} = \sum _{i=1} ^{n} X_{i}$ and let $p \geq 1$ be an integer. 

There exists a large universal constant $K$ such that for $0 < t \leq \sigma ^{2}/Kb$  we have
\begin{align} \label{Ineq: missing factor}
     \mathbb{P} \left ( S_{n} \geq t \right)  \leq \exp \left (-\frac{t^{2}}{2 \sum _{i=1}^{n} b_{i}^{2}  C_{p} \left (8tb_{i}/D_{n} ,2b_{i} , \mu^{1} _{i} , \ldots , \mu^{p}_{i} \right )} \right ) \left (\theta \left ( \frac{t}{\sigma} \right ) + \frac{Kb}{\sigma}\right ) 
     \end{align}
     where\footnote{ Note that $$\theta \left ( \frac{t}{\sigma} \right ) + \frac{Kb}{\sigma} \leq \frac{K'\sigma}{t}  $$ for some constant $K'$ so inequality (\ref{Ineq: missing factor}) provides a concentration bound of optimal order.}  $\mu^{k}_{i} = \mathbb{E}(X_{i}+ b_{i})^{k}$ for $i=1,\ldots,n$,   $k=1,\ldots,p$, $\sigma^{2} =:\mathbb{E}( \sum _{i=1}^{n} X_{i})^{2}$. 
\end{corollary}

\textbf{A limiting case.} Under the conditions of Theorem \ref{Thm: concent hof}, when $p$ tends to infinity then we need  knowledge on all the moments.  In other words, we need knowledge on the moment generating function of the random variables under consideration. The following Corollary provides an exponential bound for the case that $p$ tends to infinity. 
\begin{corollary} \label{Coro: p tends to inifinity}
Under the conditions and notations of Theorem \ref{Thm: concent hof} we have
    \begin{align*}
      \lim _{p \rightarrow \infty} C_{p}(x,b_{i},\mu^{1}_{i},\ldots,\mu^{p}_{i})= \frac{1}{b^{2}_{i}}\left (\frac{\mathbb{E}(X^{2}_{i})\exp(xX_{i}/b_{i})}{\mathbb{E}(X_{i})\exp(xX_{i}/b_{i})} \right ) ^{2}
     \end{align*}
     for all $i=1,\ldots,n$ and all $x \geq 0$. 
      Theorem \ref{Thm: concent hof} implies that for all $t>0$ we have
\begin{equation} \label{Ineq: Hoef p=infty}
     \mathbb{P}(S_{n}-\mathbb{E}(S_{n}) \geq t) \leq \exp \left (-\frac{2t^{2}}{\sum _{i=1}^{n}\left (\frac{\mathbb{E}(X^{2}_{i})\exp(4tX_{i}/D_{n})}{\mathbb{E}(X_{i})\exp(4tX_{i}/D_{n})} \right ) ^{2} } \right ).
\end{equation}
Note that inequality (\ref{Ineq: Hoef p=infty}) does not depend on $b_{i}$. 
\end{corollary}

\textbf{Examples.} We now provide examples where our results significantly improve Hoeffding's inequality.   The first example studies a sum of uniform random variables and the second example studies confidence intervals. 

\begin{example}  \label{Exam: uniform}
 (Uniform distribution). Suppose that $X_{1},\ldots,X_{n}$ are independent continuous uniform random variables on $[0,1]$, i.e., $\mathbb{P}(X_{i} \leq t) = t$ for $0 \leq t \leq 1$. In this case, a straightforward calculation shows  that 
$$\mathbb{E}X_{i}\exp(sX_{i})= \frac{\exp(s)(s-1)+1}{s^{2}} \text{  and  } \mathbb{E}X_{i}^{2}\exp(sX_{i})= \frac{\exp(s)(s^{2}-2s+2)-2}{s^{3}}.$$ 
Using Corollary \ref{Coro: p tends to inifinity} we have
$$\lim _{p \rightarrow \infty} C_{p} \left (x ,1 , \mu^{1} _{i} , \ldots , \mu^{p}_{i} \right ) = \left ( \frac{-2 + \exp(x)(2-2x+x^2)}{x(1+\exp(x)(x-1))} \right )^{2} : = C_{\infty}(x). $$ 
Using the fact that $d (X_{i}) =  \left (\mathbb{E}X_{i}^{2} / \mathbb{E}X_{i} \right )^{2} = 4/9$ inequality (\ref{Ineq: Hoef p=infty}) yields 
 \begin{equation}\label{Ineq: uniform}
     \mathbb{P}(S_{n}-\mathbb{E}(S_{n}) \geq nt) \leq \exp \left ( -\frac{2nt^{2}}{C_{\infty}(9t)} \right ).
      \end{equation}

In Figure \ref{Fig: uniform} we plot the bound given in Hoeffding's inequality (see Theorem \ref{Thm: concent hof} inequality (\ref{Ineq: Hoef bound p=1})) for $\mathbb{P}(S_{n}-\mathbb{E}(S_{n}) \geq nt)$ divided by the bound given in (\ref{Ineq: uniform}) as a function of $t$ on the interval $[0.1,0.4]$ for $n=40$. We see that the bound given in (\ref{Ineq: uniform}) significantly improves Hoeffding's bound.

\begin{figure} [ht] 
 \begin{center}
   \includegraphics[width= 0.5 \linewidth]{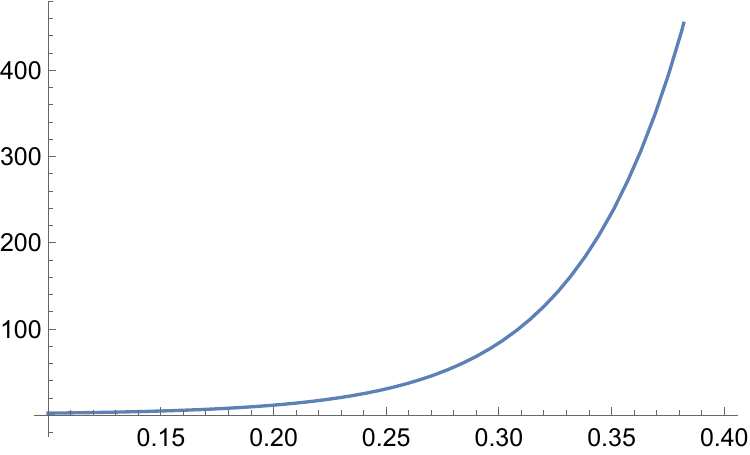}
   \caption{Comparing Hoeffding's inequality and inequality (\ref{Ineq: uniform}): $X_{i}$ is a uniform random variable on $[0,1]$ for $i=1,\ldots,40$. The plot describes the ratio of the right-hand side of inequality (\ref{Ineq: Hoef bound p=1}) (Hoeffding's inequality) to the right-hand side of inequality (\ref{Ineq: uniform}) for $n=40$. }
   \label{Fig: uniform}
    \end{center}
\end{figure}
\end{example}

\begin{example} (Confidence intervals).Suppose that we want to know how many independent and identically distributed  samples $n$ are required to acquire a $(1-\alpha)$ confidence interval around the mean $ \sum_{i=1}^{n} X_{i} /n$ of size $t$. That is, suppose that we want to find an $n$ such that
$$\delta=: \mathbb{P} \left ( \sum_{i=1}^{n} X_{i} /n \notin \left [\mathbb{E} \left( \sum_{i=1}^{n} X_{i} /n  \right )- t,\mathbb{E} \left( \sum_{i=1}^{n} X_{i} /n  \right )+ t \right ] \right )$$
is at most $\alpha$ where $X_{i}$ are independent and identically distributed random variables on $[a,b]$. Then from Remark \ref{Remark: Hoeff} part (i) for every $n$ that satisfies 
$$ n \geq \frac{\ln(2/\alpha)(b-a)^{2}\overline {C} _{p}(t,X_{i} , b-a )  } {2t^{2}}$$
we have $\delta \leq \alpha$. 
This improves the confidence interval derived by Hoeffding's inequality (which corresponds to $C_{p}=1$) by a factor of $\overline{C}_{p}$ that is given in closed-form.   This factor depends on the information that the random variable's higher moments provide. When the higher moments provide useful information and $t$ is small then this factor is significant and we need less observations to achieve a significant level of $\alpha$. 

As an example suppose that $X_{i}$ is a random variable on the interval $[0,1]$ with $\mathbb E(X_{i}) = 1/2$ and $\mathbb E(X_{i}^{2}) = 1/3$. Upon examination, it can be easily checked that $\overline{C}_{2} < 1/2$ holds for a small value of $t$ ($t<0.02$). By applying Theorem \ref{Thm: concent hof}, we find that fewer than half the samples needed when using Hoeffding's bound are required to obtain a $(1-\alpha)$ confidence interval using the generalized Hoffding's bound with $p=2$.

\end{example}

\subsection{Concentration inequalities: Bennett type inequalities} \label{Sec: concert Bennet}

In this section we derive Bennett type concentration inequalities that provide bounds on the probability that the sum of independent and bounded from above random variables differs from its expected value. The bounds depend on the random variables' first $p$ moments and are given in terms of the generalized Lambert $W$-function  \citep{scott2006general}. For real numbers $\alpha_{i}$, $i=0,\ldots,p$, $\alpha _{0} > 1$, consider the one dimensional transcendental equation: 
\begin{equation} \label{Eq: POLY-EXPO}
    \alpha_{0} - \sum _{j=1} ^{p} \alpha _{j} x^{j} = \exp(x). 
    \end{equation}
    The solutions to equation (\ref{Eq: POLY-EXPO}) are a special case of the generalized Lambert $W$-function \citep{scott2006general}. 
Because $\alpha_{0} > 1$ it is easy to see that equation (\ref{Eq: POLY-EXPO}) has a positive solution. We denote the non-empty set of positive solutions of equation (\ref{Eq: POLY-EXPO}) by $G_{p}(\alpha_{0},\ldots,\alpha_{p})$. The bounds given in Theorem \ref{Thm: concent bennett} depend on the elements of the set $G_{p}(\alpha_{0},\ldots,\alpha_{p})$ where  $\alpha_{i}$ depends on the random variables' moments.  When $p=0$ the set $G_{0}(\alpha_{0})$ consists of one element $\ln (\alpha_{0})$. When $p=1$ and assuming that $\alpha _{1} >0$, the set $G_{1}(\alpha_{0},\alpha_{1})$ consists of one element that is given in terms of the Lambert W-function. Recall that for $x \geq 0$, $y \exp (y) = x$ holds if and only if $y=W(x)$ where $W$ is the principal branch of the  Lambert W-function (see \cite{corless1996lambertw}).  Because $\alpha _{0} >1$ and assuming $\alpha _{1} >0$, the unique positive solution to the equation $\exp (x) = \alpha _{0} - \alpha_{1}x$ is given by 
$$\frac{\alpha _{0}}{\alpha_{1}} - W\left ( \frac{\exp (\alpha_{0} / \alpha_{1}) }{\alpha_{1}} \right )$$
(see \cite{corless1996lambertw}). 

Finding the positive solutions of the transcendental equation (\ref{Eq: POLY-EXPO}) for $p \geq 2$ can be done using a computer program. It involves solving an exponential polynomial equation of order $p$ that has at least one positive solution.  When the random variables have non-negative moments we show that the transcendental equation (\ref{Eq: POLY-EXPO}) has a unique positive solution (see Theorem \ref{Thm: concent bennett} part (ii)).

\begin{theorem} \label{Thm: concent bennett}
 Let $X_{1},\ldots,X_{n}$ be independent random variables on $(-\infty,b]$ for some $b>0$ and let $S_{n} = \sum _{i=1} ^{n} X_{i}$.  Let $p \geq 2$ be an integer and assume that  $X_{i} \in L^{p}$ for all $i=1,\ldots,n$. Denote $\mathbb{E}(X_{i}) = \mu ^{1}_{i}$, and assume that $\mathbb{E}(X ^{k}_{i}) \leq  \mu ^{k}_{i}$ and $0<\mathbb{E}(\max (X ^{p}_{i},0)) \leq  \mu ^{p}_{i}$  for some $\mu ^{p}_{i}$ for all $k=1,\ldots,p-1$ and all $ i=1,\ldots,n$. 
 
(i) For all $t>0$ we have
\begin{equation} \label{Ineq: Bennt concent}
\mathbb{P} (S_{n}-\mathbb{E}(S_{n}) \geq t) \leq \exp \left ( \max _{y \in  G_{p-2}(\alpha_{0},\ldots,\alpha_{p-2})} \left ( \frac{t}{b} - \left ( \frac{t}{b} + \frac{\mu^{2} }{b^{2} }\right) y +  \sum_{j=2}^{p-1} \left (\frac{\mu^{j}}{b^{j}j!} - \frac{\mu ^{j+1}}{b^{j+1}j!} \right ) y ^{j} \right ) \right )
\end{equation}
where \begin{equation*}
    \alpha_{0} = 1 + \frac{tb^{p-1}}{\mu^{p}} > 1 \text{ and } \alpha_{j} = \frac{b^{p-j-1} \mu ^{j+1}}{\mu^{p} j!} - \frac{1}{j!}  
\end{equation*} 
for all $j=1,\ldots,p-2$ and $\mu ^{k} = \sum _{i=1} ^{n} \mu _{i} ^{k} $ for all $k=1,\ldots,p$.

(ii) If $\mu^{j} \geq 0$ for every odd number $j=3,\ldots,p-1$ (for example one can choose $\mu ^{j} = \max (\sum^{n}_{i=1} \mu _{i}^{j} , 0)$) then  $G_{p-2}$ consists of one element and inequality  (\ref{Ineq: Bennt concent}) reduces to 
\begin{equation} \label{Ineq: Ben useful}
\mathbb{P} (S_{n}-\mathbb{E}(S_{n}) \geq t) \leq \exp \left (  \frac{t}{b} - \left ( \frac{t}{b} + \frac{\mu^{2} }{b^{2} }\right) y +  \sum_{j=2}^{p-1} \left (\frac{\mu^{j}}{b^{j}j!} - \frac{\mu ^{j+1}}{b^{j+1}j!} \right ) y ^{j}  \right )
\end{equation}
where $y$ is the unique element of $G_{p-2}$, i.e., $y$ is the unique positive solution of the equation $\alpha_{0} - \sum _{j=1}^{p-2} \alpha_{j}x^{j} = \exp(x)$.

(iii) Suppose that $p=2$. Then $G_{0}(\alpha_{0}) = \{\ln(\alpha_{0} ) \}$ consists of one element and inequality (\ref{Ineq: Bennt concent}) reduces to Bennett's inequality:
\begin{align}\label{Ineq: Bennett LOG p=2 CONCENT}
\begin{split}
    \mathbb{P}(S_{n}-\mathbb{E}(S_{n}) \geq t) & \leq  \exp \left (  \frac{t}{b} - \left ( \frac{t}{b} + \frac{\mu^{2} }{b^{2} }\right) \ln \left (\frac{tb}{\mu^{2}} +1 \right ) \right )  \\
    & =\exp \left (- \frac {\mu ^{2} }{b^{2}} \left ( \left (\frac{bt}{\mu^{2}} + 1 \right ) \ln \left (\frac{bt}{\mu^{2}} + 1 \right ) - \frac{bt}{\mu^{2}} \right ) \right ).
    \end{split}
\end{align} 

(iv) Suppose that $p=3$, $\alpha_{1} \neq 0$, and $\mathbb{E}(\max (X ^{3}_{i},0)) =  \mu ^{3}_{i}$ for all $ i = 1,\ldots,n$. Then $G_{1}(\alpha_{0},\alpha_{1}) =  \left \{ \frac{\alpha _{0}}{\alpha_{1}} - W\left ( \frac{\exp (\alpha_{0} / \alpha_{1}) }{\alpha_{1}} \right ) \right \}$ consists of one element and inequality (\ref{Ineq: Bennt concent}) reduces to
\begin{align} \label{Ineq: Bennett p=3 W-lamb}
\mathbb{P}(S_{n}-\mathbb{E}(S_{n}) \geq t) 
\leq \exp \left ( \frac{t}{b} - \left ( \frac{t}{b} + \frac{\mu^{2} }{b^{2} }\right) y +   \left (\frac{\mu^{2}}{2b^{2}} - \frac{\mu ^{3}}{2b^{3}} \right ) y ^{2}  \right )
\end{align}
where $y=\frac{\alpha _{0}}{\alpha_{1}} - W\left ( \frac{\exp (\alpha_{0} / \alpha_{1}) }{\alpha_{1}} \right ) $ and $W$ is the Lambert $W$-function.

\end{theorem}

The proof of Theorem \ref{Thm: concent bennett} consists of three steps. In the first step we bound the moment-generating function of a random variable $X$ that is bounded from above using the first $p$ moments of $X$. We use Theorem \ref{Coro: exp upper bound} to prove the first step. In the second step we derive an exponential bound on the moment-generating function using the elementary inequality $1+x \leq \exp(x)$ for all $x \in \mathbb{R}$. We note that in some cases this inequality is loose and and so the second step may potentially be improved (for example see \cite{jebara2018refinement} and \cite{zheng2018improved}). In the third step we apply the Chernoff bound to derive the concentration inequality.

\section{Conclusions}

We provide upper bounds on the moment-generating function of a random variable that is bounded from above  using information on the random variable's higher moments  (see Theorem \ref{Coro: exp upper bound}). Using these bounds and their convexity properties, we generalize and improve Hoeffding's inequality (see Theorem \ref{Thm: concent hof}) and Bennett's inequality (see Theorem \ref{Thm: concent bennett}) for the case that some information on the random variables' higher moments is available. Our bounds are simple to use and are given as  closed-form expressions in most cases.

\section{Proofs} \label{Sec: Proofs}

\subsection{Proofs of the results in Section \ref{Sec: bounds on the MGF}}

\begin{proof} [Proof of Theorem \ref{Coro: exp upper bound}]
Clearly Theorem \ref{Coro: exp upper bound} holds for $s=0$. Fix $s > 0$, $b>0$ and a positive integer $p$. Consider the function $g(x)=:T_{p+1}(x)/x^{p}$ on $(-\infty,\infty)$ where we define $$g(0) =: \frac{1}{p!} = \lim _{x \rightarrow 0 } g(x).$$ 

The proof proceeds with the following steps: 

\textbf{Step 1.} We have $g(x) \leq g(0)$ for all $x<0$.

\textbf{Proof of Step 1.}
First note that for $x \leq 0$ we have $T_{p} \leq 0$ if $p$ is an even number and  $T_{p} \geq 0$ if $p$ is an odd number (to see this note that $T_{1}(x) = \exp(x) \geq 0$, $T_{p}(0)=0$ and $T^{(1)}_{p} = T_{p-1}$ for all $p \geq 2$).  We now show that $g(x) \leq g(0)$ for all $x < 0$.

Suppose  that $x < 0$. If $p$ is an even number then $T_{p+1} (x)/ x ^{p} \leq 1/p!$ if and only if $T_{p+1} (x) - x ^{p}/p!  \leq 0$.  The last inequality is equivalent to $T_{p+2}(x) \leq 0$ which holds because $p$ is an even number. Similarly, if $p$ is an odd number then $T_{p+1} (x)/ x ^{p} \leq 1/p!$ if and only if $T_{p+2}(x) \geq 0$ which holds because $p$ is an odd number. 
 Thus, $g(x) \leq g(0)$ for all $x < 0$. 
 
 \textbf{Step 2.} Let $f,k:[a,b) \rightarrow \mathbb{R}$ be continuously differentiable functions such that $k^{(1)}(x) \neq 0$ for all $x \in (a,b)$. If $f^{(1)}/k^{(1)}$ is increasing on $(a,b)$ then $(f(x)-f(a))/(k(x)-k(a))$ is increasing in $x$ on $(a,b)$. 

\textbf{Proof of Step 2.} Step 2 is known as the L’Hospital rule for monotonicity. For a proof see Lemma 2.2 in \cite{anderson1993inequalities}.

 \textbf{Step 3.} The function $g$ is increasing on $(0,y)$ for all $y>0$. 
 
 \textbf{Proof of  Step 3.} Let $y>0$ and note that the function $T_{1}(x)/p!=\exp(x)/p!$ is increasing on $(0,y)$. Using Step 2 with $f(x)=\exp(x)$ and $k(x)=p!x$ implies that the function $$\frac{\exp(x)-1}{p!x}= \frac{T_{2}(x)}{p!x}$$ is increasing on $(0,y)$. Applying again Step 2 and using the facts that $T_{k+1}^{(1)} = T_{k}$ and  $T_{k}(0)=0$ for all $k=2,\ldots$ implies that the function  $T_{k}(x)/(x^{k-1}p!/(k-1)!)$ is increasing in $x$ on $(0,y)$ for all $k=2,\ldots$. Choosing $k=p+1$ shows that $g$ is increasing on $(0,y)$.
 
 \textbf{Step 4.} We have 
    \begin{equation*}
     T_{p+1}(sx) \leq  \frac{\max (x^{p},0)}{b^{p}}T_{p+1}(sb)
       \end{equation*}
       for all $ x \leq b$.
       
\textbf{Proof of Step 4.} Step 3 shows that $g$ is an increasing function on $(0,b]$. Hence, $g(x) \leq g(b)$ for all $x \in (0,b]$. Because $g$ is a continuous function we have $g(0) \leq g(b)$. Using Step 1 implies that $g(x) \leq g(b)$ for all $x \leq b$.
  
Let $x \leq b$ and assume $x \neq 0$. Multiplying each side of the inequality $g(sx) \leq g(sb)$  by the positive number $\max (x^{p},0)$ yields
   \begin{equation*}
      \frac{\max (x^{p},0)}{x^{p}}T_{p+1}(sx) \leq  \frac{\max (x^{p},0)}{b^{p}}T_{p+1}(sb).
       \end{equation*}
      Note that
       $$  T_{p+1}(sx) \leq   \frac{\max (x^{p},0)}{x^{p}}T_{p+1}(sx). $$ The last inequality holds as equality if $x > 0$ or if $p$ is an even number. If $x<0$ and $p$ is an odd number, then $T_{p+1}(sx) \leq 0$ (see Step 1), so the last inequality holds. 
     We conclude that  
    \begin{equation*}
     T_{p+1}(sx) \leq  \frac{\max (x^{p},0)}{b^{p}}T_{p+1}(sb)
       \end{equation*}
       for all $ x \leq b$.  
       
       To prove Theorem \ref{Coro: exp upper bound} apply Step 4 to conclude that
       \begin{equation*}
    \exp(sx) \leq  \frac{\max (x^{p},0)}{b^{p}}T_{p+1}(sb) + \sum_{j=0}^{p-1} \frac{s^{j}x^{j}}{j!}
       \end{equation*} 
       for all $x \leq b$. 
       Taking expectations in both sides of the last inequality proves Theorem  \ref{Coro: exp upper bound}.
\end{proof}

\begin{proof}[Proof of Proposition \ref{prop: m(p)}]
Let $p \geq 2$ be an even number. 

(i) We have  
\begin{align*}
    & \frac{\mathbb{E}(X^{p})}{b^{p}} T_{p+1}(sb)  + \mathbb{E} \left ( \sum _{j=0} ^{p-1} \frac{s^{j}X^{j}} {j!} \right )   \geq   \frac{\mathbb{E} \max (X,0)^{p+1}}{b^{p+1}} T_{p+2}(sb)  + \mathbb{E} \left ( \sum _{j=0} ^{p} \frac{s^{j}X^{j}} {j!} \right ) \\
    & \Longleftrightarrow \frac{\mathbb{E} (X^{p})}{b^{p}} \left (T_{p+1}(sb) - \frac{s^{p}b^{p}}{p!} \right ) \geq \frac{\mathbb{E} \max (X,0)^{p+1}}{b^{p+1}} T_{p+2}(sb) \\
    &  \Longleftrightarrow b{\mathbb{E} (X^{p})} \geq {\mathbb{E} \max (X,0)^{p+1}}
    \end{align*}
    which holds for a random variable $X$ on $(-\infty,b]$ and an even number $p$ because $bx^{p} \geq \max(x,0)^{p+1}$ for all $x \leq b$. 
    
    (ii) Similarly to part (i) we have $m_{X,s}(p+1) \geq m_{X,s}(p+2)$ if and only if  $b{\mathbb{E} X^{p+1}} \geq {\mathbb{E} X^{p+2}}$ 
    which holds for a non-negative random variable because $b x^{p+1} \geq x^{p+2}$ for all $0 \leq x \leq b$. 
\end{proof}

\subsection{Proofs of the results in Section \ref{Sec: concert Hoeffiding}}
\begin{proof} [Proof of Theorem \ref{Thm: concent hof}]
 We will use the following notations in proof.
Let $X$ be a random variable on $[0,b]$ with $\mathbb{P}(X>0) > 0$. Denote $\mathbb{E}(X^{k}) = \mu ^{k}$ for all $k=1,\ldots,p$. 

For every integer $p \geq 1$ we define the function 
$$ v(y,b,\mu^{1},\ldots,\mu^{p}) =:   \frac{\mu^{p}}{b^{p}} T_{p+1}(y) + 
 \sum _{j=0} ^{p-1} \frac{y^{j}\mu^{j}} {b^{j}j!}.$$
 For all $x \geq 0$ we define the function
 \begin{equation} \label{Equation Cp}
     C_{p}(x,b,\mu ^{1},\ldots, \mu ^{p})=\max _{0 \leq y \leq x}  \frac{(v^{(2)}(y,b,\mu ^{1},\ldots, \mu ^{p}))^{2}}{(v^{(1)}(y,b,\mu ^{1},\ldots, \mu ^{p}))^{2}}.
 \end{equation}

We denote by $v^{(k)}(y,b,\mu^{1},\ldots,\mu^{p})$ the $k$th derivative of $v$ with respect to its first argument. A straightforward calculation shows
\begin{equation*}
    v^{(k)}(y,b,\mu^{1},\ldots,\mu^{p}) = \frac{\mu^{p} } {b^{p}} T_{p+1-k}(y) + \sum _{j=0} ^{p-1-k} \frac{\mu^{j+k}y^{j}} {b^{j+k}j!}.  
\end{equation*}
Thus, for $p \geq 2$ we have
$$v^{(1)}(0,b,\mu^{1},\ldots,\mu^{p})= \frac{\mu^{1} }{b} > 0 \text{ and } v^{(2)}(0,b,\mu^{1},\ldots,\mu^{p})=\frac{\mu^{2}}{b^{2}}  > 0.$$ Because $v^{(2)}$ and $v^{(1)}$ are increasing in the first argument as the sum of increasing functions, we conclude that  $v^{(2)}(y,b,\mu^{1},\ldots,\mu^{p})$ and $v^{(1)}(y,b,\mu^{1},\ldots,\mu^{p})$ are positive for every $y \geq 0$. 

The proof proceeds with the following steps: 

\textbf{Step 1.} We have $\mu^{d+2}\mu^{d} \geq (\mu^{d+1})^{2}$
for every positive integer $d$. 

\textbf{Proof of Step 1.} Let $d$ be a positive integer. From the Cauchy-Schwarz inequality for the (positive) random variables $X^{(d+2)/2}$ and $X^{d/2}$  we have
$$ \mathbb{E}X^{d/2} X^{(d+2)/2} \leq \sqrt{\mathbb{E}X^{d} \mathbb{E} X^{d+2}}.$$
That is, we have $\mu^{d+2}\mu^{d} \geq (\mu^{d+1})^{2}$ which proves Step 1.

\textbf{Step 2.} For every positive integer $p$ the function $v^{(1)}(y,b,\mu^{1},\ldots,\mu^{p})$ is log-convex in $y$ on $(0,\infty)$ (i.e., $\log(v^{(1)}(y,b,\mu^{1},\ldots,\mu^{p}))$ is a convex function on $(0,\infty)$).  

\textbf{Proof of Step 2.} Fix a positive integer $p$. Let 
\begin{equation*}
   w(y)=: T_{p}(y) + \frac{b^{p}}{\mu^{p}}\sum _{j=0} ^{p-2} \frac{\mu^{j+1}y^{j}} {b^{j+1}j!} = \frac{b^{p} v^{(1)}(y,b,\mu^{1},\ldots,\mu^{p})}{\mu^{p}}.  
\end{equation*}
To prove Step 2 it is enough to prove that $w$ is log-convex on $(0,\infty)$. 
Note that 
$$ w(y) = \exp(y) + \sum _{j=0} ^{p-2} \frac{y^{j}} {j!}\beta_{j+1}$$
where 
$$\beta_{j} = \frac{b^{p-j}\mu^{j}}{\mu^{p}} - 1$$
for $j=1,\ldots,p$. 
We have $\beta_{j} \geq 0$ for all $j=1,\ldots,p$. To see this note that $x^{j}b^{p-j} \geq x^{p}$ for all $x \in [0,b]$ so taking expectations implies that $\beta_{j} \geq 0$.

$w$ is log-convex on $(0,\infty)$ if and only if 
$w^{(1)}/w$ is increasing on $(0,\infty)$. For every integer $k=0,\ldots,p-2$ define the function 
$$ w_{k}(y) = \exp(y) + \sum _{j=0} ^{k} \frac{y^{j}} {j!}\beta_{p-1+j-k}$$
and note that $w_{p-2} = w$. By construction we have $w_{k}^{(1)} = w_{k-1}$ We now show that $w_{k}$ is log-convex on $(0,\infty)$ for all $k=0,\ldots,p-2$. The proof is by induction.  

For $k=0$ the function 
$$ \frac{w^{(1)}_{0}(y)}{w_{0}(y)} = \frac{\exp(y)}{\exp(y) + \beta_{p-1}} $$
is increasing because $ \beta_{p-1} \geq 0$ and the function $x/(x+d)$ is increasing in $x$ on $[0,\infty)$ when $d \geq 0$. We conclude that
the function $w_{0} (y) = \exp(y) + \beta_{p-1}$ is log-convex on $(0,\infty)$. 

Assume that $w_{k}$ is log-convex on $(0,\infty)$ for some integer $0 \leq k \leq p-3$. We show that $w_{k+1}$ is log-convex on $(0,\infty)$. Log-convexity of  $w_{k}$  implies that the function 
\begin{equation} \label{Log-convex1}
    \frac{w^{(1)}_{k}(y)}{w_{k}(y)} = \frac{\exp(y) + \sum _{j=0} ^{k-1} \frac{y^{j}} {j!}\beta_{p+j-k}}{\exp(y) + \sum _{j=0} ^{k} \frac{y^{j}} {j!}\beta_{p-1+j-k}}
\end{equation}
is increasing on $(0,\infty)$. Using the fact that $w^{(1)}_{k+1}=w_{k}$ and applying Step 2 in the proof of Theorem \ref{Coro: exp upper bound} we conclude that the function 
$$ m(y)=: \frac{w_{k}(y) - w_{k}(0)}{w_{k+1}(y) - w_{k+1}(0)} = \frac{\exp(y) + \sum _{j=0} ^{k} \frac{y^{j}} {j!}\beta_{p-1+j-k} - (1+\beta_{p-1-k})}{\exp(y) + \sum _{j=0} ^{k+1} \frac{y^{j}} {j!}\beta_{p-2+j-k} - (1+\beta_{p-2-k})}$$ 
is increasing on $(0,\infty)$. Thus, $m^{(1)}(y) \geq 0$ for all $y \in (0,\infty)$. That is, 
\begin{equation} \label{Log-convex2}
w_{k+1}^{(2)}(y) w_{k+1}(y) -w_{k+1}^{(2)}(y)(1+\beta_{p-2-k}) \geq (w_{k+1}^{(1)}(y))^{2} - w_{k+1}^{(1)}(y) (1+\beta_{p-1-k}) 
\end{equation}
for all $y \in (0,\infty)$. 
%where  
%$$ c(y)= \left (\exp(y) + \sum _{j=0} ^{k-1} \frac{y^{j}} {j!}\beta_{p+j-k} \right ) \left (\exp(y) + \sum _{j=0} ^{k+1} \frac{y^{j}} {j!}\beta_{p-2+j-k}\right ) = w_{k+1}^{(2)} w_{k+1} $$
%and $$ d(y)= \left(\exp(y) + \sum _{j=0} ^{k} \frac{y^{j}} {j!}\beta_{p-1+j-k} \right )^{2} = (w_{k+1}^{(1)})^{2}.$$
We now show that $w_{k+1}^{(2)}(y) w_{k+1}(y) \geq (w_{k+1}^{(1)}(y))^{2}$. Because $w_{k+1}^{(2)}/w_{k+1}^{(1)}$ is increasing and positive (see (\ref{Log-convex1})) we have  
\begin{equation}\label{Log-convex3} \frac{w_{k+1}^{(2)}(y)}{w_{k+1}^{(1)}(y)}(1+\beta_{p-2-k}) \geq (1+\beta_{p-1-k})
\end{equation}
for all $y \in (0,\infty)$ if the last inequality holds for $y=0$, i.e., if
\begin{align*} 
(1+\beta_{p-k}) (1+ \beta_{p-2-k}) \geq (1+ \beta_{p-1-k})^{2} & \Longleftrightarrow  \left (\frac{b^{k}\mu^{p-k}}{\mu^{p}} \right ) \left (\frac{b^{k+2}\mu^{p-2-k}}{\mu^{p}} \right ) \geq \left (\frac{b^{k+1}\mu^{p-k-1}}{\mu^{p}} \right )^{2} \\
& \Longleftrightarrow \mu^{p-k} \mu^{p-k-2} \geq (\mu^{p-k-1})^{2}
\end{align*}
which holds from Step 1. We conclude that inequality (\ref{Log-convex3}) holds. Using inequality (\ref{Log-convex2}) we have 
$$ w_{k+1}^{(2)}(y) w_{k+1}(y) - (w_{k+1}^{(1)}(y))^{2} \geq w_{k+1}^{(2)}(y)(1+\beta_{p-2-k}) - w_{k+1}^{(1)}(y)(1+\beta_{p-1-k}) \geq 0. $$
That is, $w_{k+1}^{(2)}(y) w_{k+1}(y) \geq (w_{k+1}^{(1)}(y))^{2}$ for all $y \in (0,\infty)$. We conclude that $w^{(1)}_{k+1}/w_{k+1}$ is increasing on $(0,\infty)$, i.e., $w_{k+1}$ is log-convex. This shows that $w_{k}$ is log-convex for all $k=0,\ldots,p-2$. In particular, $w_{p-2}=:w$ is log-convex which proves Step 2.

\textbf{Step 3.} 
We have
\begin{equation*}
    C_{p}(x,b,\mu^{1},\ldots,\mu^{p}) = \left ( \frac{\exp(x) + \sum _{j=0} ^{p-3} \frac{x^{j}} {j!} \left (  \frac{b^{p-j-2}\mu^{j+2}}{\mu^{p}} - 1 \right )}{\exp(x) + \sum _{j=0} ^{p-2} \frac{x^{j}} {j!} \left (  \frac{b^{p-j-1}\mu^{j+1}}{\mu^{p}} - 1 \right )} \right )^{2}. 
\end{equation*}
for all $x \geq 0$ where $C_{p}$ is given in Equation (\ref{Equation Cp}). 

\textbf{Proof of Step 3.} Let $x \geq 0$. From Step 2 the function $v^{(1)}(y,b,\mu ^{1},\ldots, \mu ^{p})b^{p}/\mu^{p}:=w(y)$ is log-convex on $(0,x)$ where $w$ is defined in the proof of Step 2. This implies that
$$   C_{p}(x,b,\mu ^{1},\ldots, \mu ^{p})=\max _{0 \leq y \leq x}  \frac{(v^{(2)}(y,b,\mu ^{1},\ldots, \mu ^{p}))^{2}}{(v^{(1)}(y,b,\mu ^{1},\ldots, \mu ^{p}))^{2}} = \max _{0 \leq y \leq x} \left (\frac{w^{(1)}(y)} {w(y)} \right )^{2} = \left (\frac{w^{(1)}(x)} {w(x)} \right )^{2} $$
which proves Step 3.

\textbf{Step 4.}  For all $x \geq 0$ we have
 \begin{equation*}
      \max _{ 0 \leq y \leq x} \left ( \frac{v^{(2)}(y,b,\mu ^{1},\ldots, \mu ^{p})}{v(y,b,\mu ^{1},\ldots, \mu ^{p})} - \frac{(v^{(1)}(y,b,\mu ^{1},\ldots, \mu ^{p}))^{2}}{(v(y,b,\mu ^{1},\ldots, \mu ^{p}))^{2}}  \right ) \leq \frac{1}{4} C_{p}(x,b,\mu ^{1},\ldots, \mu ^{p}). 
 \end{equation*}
 \textbf{Proof of Step 4.} For all $x > 0$ we have
\begin{align*}
  & \max _{ 0 \leq y \leq x} \left ( \frac{v^{(2)}(y,b,\mu ^{1},\ldots, \mu ^{p})}{v(y,b,\mu ^{1},\ldots, \mu ^{p})} - \frac{(v^{(1)}(y,b,\mu ^{1},\ldots, \mu ^{p}))^{2}}{(v(y,b,\mu ^{1},\ldots, \mu ^{p}))^{2}}  \right )\\
& = \max _{ 0 \leq y \leq x}  \left ( \frac{v^{(2)}(y,b,\mu ^{1},\ldots, \mu ^{p})}{v(y,b,\mu ^{1},\ldots, \mu ^{p})} \left (1 - \frac{v^{(2)}(y,b,\mu ^{1},\ldots, \mu ^{p})(v^{(1)}(y,b,\mu ^{1},\ldots, \mu ^{p}))^{2}}{v(y,b,\mu ^{1},\ldots, \mu ^{p})(v^{(2)}(y,b,\mu ^{1},\ldots, \mu ^{p}))^{2}} \right ) \right )  \\ 
& \leq \max _{0 \leq y \leq x} \left ( \frac{v^{(2)}(y,b,\mu ^{1},\ldots, \mu ^{p})}{v(y,b,\mu ^{1},\ldots, \mu ^{p})} \left (1 - \frac{v^{(2)}(y,b,\mu ^{1},\ldots, \mu ^{p})}{v(y,b,\mu ^{1},\ldots, \mu ^{p}) C_{p}(x,b,\mu ^{1},\ldots, \mu ^{p})} \right ) \right )  \\
& \leq \frac{1}{4} C_{p}(x,b,\mu ^{1},\ldots, \mu ^{p}).
\end{align*}
 The first inequality follows from the definition of $C_{p}$ and because $v^{(2)} > 0$ and $v^{(1)} > 0$. The second inequality follows from the elementary inequality $x(1-x/z) \leq 0.25z$ for all $z>0$ and $x>0$.

\textbf{Step 5.} For $s \geq 0$ we have 
\begin{align*}
    \mathbb{E} \exp (s(X-\mathbb{E}(X)) \leq \exp \left ( \frac{s^2 b^{2}C_{p}(sb,b,\mu ^{1}, \ldots , \mu ^{p} ) }{8} \right )
\end{align*} 
\textbf{Proof of Step 5.} From Theorem \ref{Coro: exp upper bound} for all $s \geq 0$ we have 
\begin{align*}\mathbb{E}\exp (sX) & \leq  \frac{\mathbb{E} X ^{p} } {b^{p}} T_{p+1}(sb) + \mathbb{E} \left ( \sum _{j=0} ^{p-1} \frac{s^{j}X^{j}} { j!} \right ) \\
& = v(y,b,\mu^{1},\ldots,\mu^{p})
\end{align*} 
where $y=sb \geq 0$.  
Define the function
\begin{align*}
    g(y) & = \ln \left (  v(y,b,\mu^{1},\ldots,\mu^{p})
 \right ).
\end{align*}
Clearly $v$ is a positive function so the function $g :\mathbb{R}_{ +} \rightarrow \mathbb{R}$ is well defined. Note that $\mathbb{E}\exp (sX) \leq \exp (g(y))$. 
Recall that
 $v^{(1)}(0,b,\mu^{1},\ldots,\mu^{p})= \mathbb{E}(X)/b$. 
Because $v(0,b,\mu^{1} , \ldots , \mu^{p}) = 1$ we have $g(0)  =\ln (1) =0$. We have
\begin{align*}
g^{(1)}(y) =  \frac{v^{(1)}(y,b, \mu^{1},\ldots,\mu^{p})}{v(y,b,\mu^{1},\ldots,\mu^{p})}.
\end{align*}
Thus, $g^{(1)}(0) = \mathbb{E}(X)/b$.
Differentiating again yields
\begin{align*}
g^{ (2) }(y) & = 
\frac{v^{(2)}(y,b,\mu ^{1},\ldots, \mu ^{p})}{v(y,b,\mu ^{1},\ldots, \mu ^{p})} - \frac{(v^{(1)}(y,b,\mu ^{1},\ldots, \mu ^{p}))^{2}}{(v(y,b,\mu ^{1},\ldots, \mu ^{p}))^{2}}. 
\end{align*}

From Taylor's theorem for all $y \geq 0$ there exists a $z \in [0,y]$ such that $g(y) = g(0) + yg^{(1)}(0) + 0.5 y^2 g^{(2)}(z) $. Thus, using the fact that $y=sb$ we have 
\begin{align*}
    g(y) = g(0) + yg^{(1)}(0) + 0.5 y^2 g^{(2)}(z) = s\mathbb{E}(X)+ 0.5 s^2 b^{2} g^{(2)}(z) \leq  s\mathbb{E}(X) + 0.5  s^2 b^{2}  V(sb,b,\mu ^{1}, \ldots , \mu ^{p} ) 
\end{align*}
where 
$$V(y,b,\mu ^{1}, \ldots , \mu ^{p} ) = \sup _{0 \leq z \leq y} g^{(2)} (z). $$ 
Using $\mathbb{E}\exp(sX) \leq \exp (g(y))$ and Step 4 imply
\begin{align*}
    \mathbb{E} \exp (s(X-\mathbb{E}(X)) & \leq  \exp \left ( \frac{ s^2 b^{2}V(sb,b,\mu ^{1}, \ldots , \mu ^{p} ) } {2} \right ) \\
    & \leq \exp \left ( \frac{s^2 b^{2}C_{p}(sb,b,\mu ^{1}, \ldots , \mu ^{p} ) }{8} \right ).
\end{align*} 
 
\textbf{Step 6.} For all $t>0$ we have 
\begin{equation*}
     \mathbb{P}(S_{n}-\mathbb{E}(S_{n}) \geq t) \leq \exp \left (- \frac{t^{2}}{2 \sum_{i=1}^{n} b^{2} _{i} C_{p} \left (\frac{4tb_{i}}{\sum _{j=1} ^{n} d(X_{j}) },b_{i}, \mu  ^{1}_{i}, \ldots , \mu ^{p}_{i} \right ) } \right ).
\end{equation*}

\textbf{Proof of Step 6.}
Using independence, Step 5, and Markov's inequality, a standard argument shows that: 
\begin{align*}
    \mathbb{P}(S_{n}-\mathbb{E}(S_{n}) \geq t) & \leq \exp(-st) \mathbb{E} \exp (s(S_{n}-\mathbb{E}(S_{n})) \\
    & = \exp(-st)  \prod_{i=1}^{n} \mathbb{E} \exp (s(X_{i}-\mathbb{E}(X_{i})) \\
    & \leq \exp(-st)  \prod_{i=1}^{n}   \exp \left ( \frac{s^2 b^{2}_{i} C_{p}(sb_{i},b_{i},\mu ^{1}_{i}, \ldots , \mu ^{p}_{i} ) }{8} \right )\\
    & = \exp \left (-st + \frac{s^{2} } {8} \sum_{i=1}^{n} b^{2}_{i}  C_{p} (sb_{i},b_{i}, \mu  ^{1}_{i}, \ldots , \mu ^{p}_{i} ) \right ).  
\end{align*}
Let
$$ s = \frac{4t}{\sum _{i=1} ^{n} b^{2}_{i} C_{p} \left (\frac {4tb_{i} } { \sum _{j=1}^{n} d(X_{j})},b_{i}, \mu  ^{1}_{i}, \ldots , \mu ^{p}_{i} \right )}.$$
Note that
$$C_{p}(0,b_{i},\mu ^{1}_{i}, \ldots , \mu ^{p}_{i} ) =  \left ( \frac{\mu^{2}_{i}}{\mu^{1}_{i} b_{i}}  \right )^{2} = \frac {d(X_{i})} {b_{i}^{2}}.$$ 
Because $C_{p}$ is increasing in the first argument for all $y \geq 0$, we have  
\begin{equation} \label{Eq: Vp proof H}
    C_{p}(y,b_{i},\mu ^{1}_{i}, \ldots , \mu ^{p}_{i} ) \geq  \frac {d(X_{i})} {b_{i}^{2}} > 0.
\end{equation}
Thus,
 $$ sb_{i} = \frac{4tb_{i}}{\sum _{i=1} ^{n} b^{2} _{i}  C_{p} \left (\frac {4tb_{i} } { \sum _{j=1}^{n} d(X_{j})},b_{i}, \mu  ^{1}_{i}, \ldots , \mu ^{p}_{i} \right )} \leq \frac{4tb_{i}}{\sum _{j=1}^{n} d(X_{j})}.$$
 Using again the fact that $C_{p}$ is increasing in the first argument implies 
 \begin{align*}
    -st + \frac{ s^{2} } {8} \sum_{i=1}^{n}  b^{2} _{i} C_{p}(sb_{i},b_{i} , \mu  ^{1}_{i}, \ldots , \mu ^{p}_{i} )   & \leq  -st + \frac{s^{2} }{8} \sum_{i=1}^{n} b^{2} _{i}  C_{p} \left (\frac{4tb_{i}}{\sum _{j=1} ^{n} d(X_{j}) },b_{i},\mu  ^{1}_{i}, \ldots , \mu ^{p}_{i} \right ) \\
    & = - \frac{2t^{2}}{\sum_{i=1}^{n} b^{2} _{i} C_{p} \left (\frac{4tb_{i}}{\sum _{j=1} ^{n} d(X_{j})  },b_{i},\mu  ^{1}_{i}, \ldots , \mu ^{p}_{i} \right ) }.
\end{align*}
We conclude that
\begin{equation} \label{Hof theorem Conc1}
    \mathbb{P}(S_{n}-\mathbb{E}(S_{n}) \geq t) \leq \exp \left (- \frac{2t^{2}}{\sum_{i=1}^{n} b^{2} _{i} C_{p} \left (\frac{4tb_{i}}{\sum _{j=1} ^{n} d(X_{j}) },b_{i}, \mu  ^{1}_{i}, \ldots , \mu ^{p}_{i} \right ) } \right ).
\end{equation}
which proves Step 6.

Combining Steps 3 and 6 proves part (i). 

(ii) Let $X$ be a random variable on $[0,b]$. Denote $\mathbb{E}(X^{k}) = \mu ^{k}$ for all $k=1,\ldots,p$.  Clearly $0 < C_{p}$ because $v^{(2)}$ and $v^{(1)}$ are positive functions (see part (i)).

 We show that $v^{(2)}(y,b,\mu ^{1},\ldots, \mu ^{p}) \leq v^{(1)}(y,b,\mu ^{1},\ldots, \mu ^{p})$ for all $y \geq 0$. 

We have $v^{(2)}(y,b,\mu ^{1},\ldots, \mu ^{p}) \leq v^{(1)}(y,b,\mu ^{1},\ldots, \mu ^{p})$ if and only if 
$$ \frac{\mu ^{p}}{b^{p}}T_{p-1}(y) +  \sum _{j=0} ^{p-3} \frac{y^{j}\mu^{j+2}} {b^{j+2}j!} \leq  \frac{\mu ^{p}}{b^{p}}T_{p}(y) +  \sum _{j=0} ^{p-2} \frac{y^{j}\mu^{j+1}} {b^{j+1}j!}. $$
The last inequality holds if and only if 
\begin{align*}
    & \frac{ \mu ^{p} y^{p-2}}{b^{p}(p-2)!}  +   \sum _{j=0} ^{p-3} \frac{y^{j}\mu ^{j+2}} {b^{j+2}j!}  -   \sum _{j=0} ^{p-2} \frac{y^{j}\mu^{j+1}} {b^{j+1}j!}  \leq 0 \\ 
    & \iff   \sum _{j=0} ^{p-2} \frac{y^{j}\mu^{j+2}} {b^{j+2}j!} -   \sum _{j=0} ^{p-2} \frac{y^{j}\mu ^{j+1}} {b^{j+1}j!}  \leq 0. 
\end{align*}
To see that the last inequality holds let $0 \leq x \leq b$. We have  $b^{j+1}x^{j+2} \leq x^{j+1}b^{j+2}$. Taking expectations and multiplying by $y^{j} / j!$ show that
\begin{equation*}
\frac{y^{j}\mu^{j+2}} {b^{j+2}j!} \leq \frac{y^{j}\mu^{j+1}} {b^{j+1}j!}
\end{equation*}
for all $1 \leq j \leq p-2$ and all $y \geq 0$. 
We conclude that $v^{(2)}(y,b,\mu ^{1},\ldots, \mu ^{p}) \leq v^{(1)}(y,b,\mu ^{1},\ldots, \mu ^{p})$ for all $y \geq 0$. Thus, 
$$   C_{p}(x,b,\mu ^{1},\ldots, \mu ^{p})=\max _{0 \leq y \leq x}  \frac{(v^{(2)}(y,b,\mu ^{1},\ldots, \mu ^{p}))^{2}}{(v^{(1)}(y,b,\mu ^{1},\ldots, \mu ^{p}))^{2}} \leq 1$$
which immediately implies that inequality (\ref{Ineq: Hoef bound}) is tighter then Hoeffding's inequality which corresponds to $C_{p}=1$ (when $p=1$ the argument above shows that $v^{2} = v^{1}$ so $C_{1} = 1$ and we derive Hoeffding's inequality (\ref{Ineq: Hoef bound p=1})). 
\end{proof} 

\begin{proof}[Proof of Corollary \ref{Coro: Special case}]
Under the Corollary's assumption we have $4t / d(X_{i}) \leq c$. Because the function $C_{p}$ is increasing in the first argument (see the proof of Theorem \ref{Thm: concent hof}) we have 
$$( I_{p}(X_{i},c) )^{2} = \frac{1} {C_{p}(c,1,\mathbb{E}(X_{i}),\ldots,\mathbb{E}(X^{p}_{i}))} \leq   \frac{1} {C_{p}(4t/d(X_{i}),1,\mathbb{E}(X_{i}),\ldots,\mathbb{E}(X^{p}_{i}))}.$$
Combining the inequality above and Remark \ref{Remark: Hoeff} part (ii) proves the result. 
\end{proof}

\begin{proof}[Proof of Corollary \ref{Coro: Talag}]
Theorem 3.3 in \cite{talagrand1995missing} shows that there exists a universal constant $K$ such that for all $t \leq \sigma ^{2}/Kb$ we have 
$$ \mathbb{P} \left ( S_{n} \geq t \right)  \leq \inf _{s \geq 0}\exp   \left ( \sum_{i=1}^{n} \ln \mathbb{E} \exp(sX_{i}) - st \right ) \left (\theta \left ( \frac{t}{\sigma} \right ) + \frac{Kb}{\sigma}\right ) . $$ The result now follows from Steps 5 and 6 in the proof of Theorem \ref{Thm: concent hof}. 
\end{proof}

\begin{proof}[Proof of Corollary \ref{Coro: p tends to inifinity}]
 Let $Z$ be a random variable on $[0,b]$. Denote $\mathbb{E}(Z^{k}) = \mu ^{k}$ for all $k=1,\ldots$ and let $y \in [0,x]$ for some $x \geq 0$. 

First note that $0 \leq \lim _{p \rightarrow \infty}   b^{-p}\mu^{p}T_{p}(y) \leq  \lim _{p \rightarrow \infty} T_{p}(y) = 0$. 
Because $Z$ is a random variable on $[0,b]$ we can use the bounded convergence theorem to conclude that
\begin{align*}
\lim _{p \rightarrow \infty} v^{(1)}(y,b,\mu ^{1},\ldots, \mu ^{p}) =  \lim_{p \rightarrow \infty} \left ( \frac{\mu ^{p}}{b^{p}}T_{p}(y) +  \sum _{j=0} ^{p-2} \frac{y^{j}\mu^{j+1}} {b^{j+1}j!} \right )  = b^{-1}\mathbb{E}Z \exp(yZ/b). 
\end{align*}
Similarly, 
\begin{align*}
\lim _{p \rightarrow \infty} v^{(2)}(y,b,\mu ^{1},\ldots, \mu ^{p}) = \lim _{p \rightarrow \infty} \left (  \frac{\mu ^{p}}{b^{p}}T_{p-1}(y) +  \sum _{j=0} ^{p-3} \frac{y^{j}\mu^{j+2}} {b^{j+2}j!} \right ) = b^{-2}\mathbb{E}Z^{2} \exp(yZ/b). 
\end{align*}
We conclude that 
$$ \lim _{p \rightarrow \infty} \left ( \frac{ v^{(2)}(y,b,\mu ^{1},\ldots, \mu ^{p})} {v^{(1)}(y,b,\mu ^{1},\ldots, \mu ^{p})} \right ) ^{2} = \left (  \frac{b^{-2}\mathbb{E}Z^{2} \exp(yZ/b)}{b^{-1}\mathbb{E}Z \exp(yZ/b)} \right )^{2} = b^{-2}   \left (  \frac{\mathbb{E}Z^{2} \exp(yZ/b)}{\mathbb{E}Z \exp(yZ/b)} \right )^{2}$$

Using Step 3 in the proof of Theorem \ref{Thm: concent hof} part (i) yields 
\begin{align*}
    \lim _{p \rightarrow \infty} C_{p} (x,b_{i},\mu^{1},\ldots,\mu^{p}) & = \lim _{p \rightarrow \infty} \max _{0 \leq y \leq x} \left ( \frac{ v^{(2)}(y,b,\mu ^{1},\ldots, \mu ^{p})} {v^{(1)}(y,b,\mu ^{1},\ldots, \mu ^{p})} \right ) ^{2} \\
    & =   \lim _{p \rightarrow \infty} \left ( \frac{ v^{(2)}(x,b,\mu ^{1},\ldots, \mu ^{p})} {v^{(1)}(x,b,\mu ^{1},\ldots, \mu ^{p})} \right )^{2} \\
    & =  b^{-2}   \left (  \frac{\mathbb{E}Z^{2} \exp(xZ/b)}{\mathbb{E}Z \exp(xZ/b)} \right )^{2}
\end{align*}
which proves the result. 
\end{proof}

\subsection{Proofs of the results in Section \ref{Sec: concert Bennet}}

\begin{proof}[Proof of Theorem \ref{Thm: concent bennett}]
(i) Let $ s \geq 0$ and let $p \geq 2$ be an integer.  We first assume that $b=1$ so that $X_{i}$  is a random variable on $(-\infty,1]$ for all $i=1,\ldots,n$.  

For any random variable $X_{i}$ on $(-\infty,1]$ we have
\begin{align*}
   \mathbb{E}\exp (sX_{i})  & \leq \mu ^{p}_{i} \left (\exp(s) - \sum _{j=0} ^{p-1} \frac{s^{j}}{j!}  \right ) + 1 + \sum_{j=1}^{p-1} \frac{s^{j}\mu ^{j}_{i}}{j!} \\
    & \leq \exp \left (\mu ^{p}_{i}\left (\exp(s) - \sum _{j=0} ^{p-1} \frac{s^{j}}{j!}  \right ) + \sum_{j=1}^{p-1} \frac{s^{j}\mu ^{j}_{i}}{j!} \right ) \\
    & = \exp \left (\mu ^{p}_{i} T_{p+1}(s) + \sum_{j=1}^{p-1} \frac{s^{j}\mu ^{j}_{i}}{j!} \right ).
\end{align*}
 The first inequality follows from Theorem \ref{Coro: exp upper bound} and the fact that $T_{p+1}(s) \geq 0$ for $s \geq 0$. The second inequality follows from the elementary inequality $1 + x \leq e^{x}$ for all $x \in \mathbb{R}$. Thus, 
\begin{equation*}
     \mathbb{E}\exp (s(X_{i}-\mu ^{1}_{i})) \leq  \exp \left (\mu ^{p}_{i} T_{p+1}(s) + \sum_{j=2}^{p-1} \frac{s^{j}\mu ^{j}_{i}}{j!}  \right ) 
\end{equation*}
and
\begin{align*}
\prod_{i=1}^{n} \mathbb{E} \exp (s(X_{i}-\mathbb{E}(X_{i})) & \leq \prod_{i=1}^{n}  \exp \left (\mu ^{p}_{i} T_{p+1}(s) + \sum_{j=2}^{p-1} \frac{s^{j}\mu ^{j}_{i}}{j!} \right ) \\
& = \exp \left (\mu ^{p}T_{p+1}(s) + \sum_{j=2}^{p-1} \frac{s^{j}\mu ^{j}}{j!} \right ).
\end{align*}

From the Chernoff bound and the fact that $X_{1},\ldots,X_{n}$ are independent random variables, for all $t > 0$, we have 
\begin{align*}
\mathbb{P}(S_{n}-\mathbb{E}(S_{n}) \geq t) & \leq  \inf _{s \geq 0} \exp(-st) \mathbb{E} \exp (s(S_{n}-\mathbb{E}(S_{n})) \\
    & = \inf _{s \geq 0} \exp(-st)  \prod_{i=1}^{n} \mathbb{E} \exp (s(X_{i}-\mathbb{E}(X_{i})) \\
    & \leq \inf _{s \geq 0}  \exp \left (-st + \sum_{j=2}^{p-1} \frac{s^{j}\mu ^{j}}{j!} + \mu ^{p}  T_{p+1}(s)  \right ) \\ 
    & =   \exp \left ( - \sup _{s \geq 0} \left (  st - \sum_{j=2}^{p-1} \frac{s^{j}\mu ^{j}}{j!} -  \mu ^{p}  T_{p+1}(s)  \right )  \right ) \\ 
    & = \exp \left (- \mu^{p} \sup _{x \geq 0} h_{p} (x,t, \mu ^{2},\ldots,\mu^{p} ) \right )
\end{align*}
where 
\begin{align*}
    h_{p}(x,t,\mu ^{2},\ldots,\mu^{p} ) & =  \frac{t}{\mu^{p}}x -  \frac{1}{\mu ^{p}} \sum_{j=2}^{p-1} \frac{x^{j}\mu ^{j}}{ j!} - T_{p+1}(x)  \\
    & = 1  + \left ( \frac{t}{\mu^{p}} + 1\right ) x - \sum_{j=2}^{p-1} \left (\frac{\mu ^{j}}{\mu^{p} j!} - \frac{1}{j!}   \right ) x^{j} -  \exp(x).
    \end{align*} 
    Because $h_{p}$ is continuous, $h_{p}(0,t,\mu ^{2},\ldots,\mu^{p}) = 0$, and $\lim _{x \rightarrow \infty} h_{p}(x,t,\mu ^{2},\ldots,\mu^{p}) = -\infty$, the function $h_{p}$ has a maximizer. Let $h^{(j)}_{p}$ the $j$th derivative of $h_{p}$ with respect to $x$.

Note that 
\begin{align*}
    h^{(1)}_{p}(x,t,\mu ^{2},\ldots,\mu^{p} ) & =  \frac{t}{\mu^{p}} + 1 -  \sum_{j=1}^{p-2} \left (\frac{\mu ^{j+1}}{\mu ^{p} j!} - \frac{1}{j!}   \right ) x^{j} -  \exp(x)  \\
    & = \alpha_{0} - \sum _{j=1}^{p-2} \alpha_{j} x^{j} - \exp(x)
    \end{align*}
Thus, $h^{(1)}_{p} (0,t,\mu ^{2},\ldots,\mu^{p} ) = \alpha_{0} - \exp(0) > 0 $ and $h^{(1)}_{p} (x,t,\mu ^{2},\ldots,\mu^{p} ) < 0$ for all $x \geq \overline{x}$ for some large $\overline{x}$. Because $h^{(1)}_{p}$ is continuous we conclude that a maximizer $y$ of $h_{p}$ on $[0,\infty)$ satisfies $h^{(1)}_{p} (y,t,\mu ^{2},\ldots,\mu^{p} ) = 0$, that is, $y \in G_{p-2}(\alpha_{0},\ldots,\alpha_{p-2})$.  Plugging $y$ into $h_{p}$ yields
\begin{align*}
   & \left ( \frac{t}{\mu^{p}} + 1\right ) y -  \sum_{j=2}^{p-1} \left (\frac{\mu ^{j}}{\mu^{p} j!} - \frac{1}{j!}   \right ) y^{j} + 1 -  \exp(y) \\
    & =  \left ( \frac{t}{\mu^{p}} + 1\right ) y - \sum_{j=2}^{p-1} \left (\frac{\mu ^{j}}{\mu^{p} j!} - \frac{1}{j!}   \right ) y^{j} -  \frac{t}{\mu^{p}} +  \sum_{j=1}^{p-2} \left (\frac{\mu ^{j+1}}{ \mu^{p} j!} - \frac{1}{j!}   \right ) y^{j} \\  
    & = -  \frac{t}{\mu^{p}} + \left ( \frac{t}{\mu^{p}} + \frac{\mu^{2}}{\mu^{p}}\right ) y -  \frac{1}{\mu^{p}} \sum_{j=2}^{p-1} \left (\frac{\mu ^{j}}{ j!} - \frac{\mu^{j+1}}{j!}   \right ) y^{j}.
\end{align*}
In the first equality we used the fact that $h^{(1)}_{p} (y,t,\mu ^{2},\ldots,\mu^{p} ) = 0$. 
Thus, 
\begin{align}\label{Bennett Proof1}
\begin{split}
    \mathbb{P}(S_{n}-\mathbb{E}(S_{n}) \geq t) & \leq \exp \left ( - \mu^{p} \left ( -  \frac{t}{\mu^{p}} + \left ( \frac{t}{\mu^{p}} + \frac{\mu^{2}}{\mu^{p}}\right ) y -  \frac{1}{\mu^{p}} \sum_{j=2}^{p-1} \left (\frac{\mu ^{j}}{ j!} - \frac{\mu^{j+1}}{j!}   \right ) y^{j}\right )  \right ) \\
    & = \exp \left ( \max _{y \in  G_{p-2}(\alpha_{0},\ldots,\alpha_{p-2})} \left ( t - \left ( t + \mu^{2}\right ) y + \sum_{j=2}^{p-1} \left (\frac{\mu ^{j}}{ j!} - \frac{\mu^{j+1}}{j!}   \right ) y^{j} \right ) \right ) 
    \end{split}
\end{align}
which proves part (i) for the case that $b=1$. Now suppose that $b \neq 1$ and $X_{i} \leq b$ for some $b>0$. Define the random variable $Y_{i} = X_{i}/b$ and note that $Y_{i} \leq 1$ and $ \mathbb{E}Y_{i}^{k} \leq \mu ^{k}_{i}/b^{k}$. Thus, we can apply inequality (\ref{Bennett Proof1}) for the random variables $Y_{1},\ldots,Y_{n}$ to conclude that for all $t>0$ we have
\begin{align*}
\begin{split}
& \mathbb{P} (S_{n}-\mathbb{E}(S_{n}) \geq t) = \mathbb{P} \left (\sum _{i=1} ^{n} Y_{i} -\mathbb{E}\left (\sum _{i=1} ^{n} Y_{i} \right ) \geq \frac{t}{b} \right ) \\
& \leq \exp \left ( \max _{y \in  G_{p-2}(\alpha_{0},\ldots,\alpha_{p-2})} \left ( \frac{t}{b} - \left ( \frac{t}{b} + \frac{\mu^{2} }{b^{2} }\right) y +  \sum_{j=2}^{p-1} \left (\frac{\mu^{j}}{b^{j}j!} - \frac{\mu ^{j+1}}{b^{j+1}j!} \right ) y ^{j} \right ) \right )
\end{split}
\end{align*}
where \begin{equation*}
    \alpha_{0} = 1 + \frac{tb^{p-1}}{\mu^{p}} > 1 \text{ and } \alpha_{j} = \frac{b^{p-j-1} \mu ^{j+1}}{\mu^{p} j!} - \frac{1}{j!}  
\end{equation*} 
for all $j=1,\ldots,p-2$. This proves part (i).

(ii) Suppose for simplicity that $b=1$ (as in part (i) part (ii) holds for any $b>0$ when it holds for $b=1$). Note that
\begin{equation*}
     h^{(1)}_{p}(x,t,\mu ^{2},\ldots,\mu^{p} )  =  \frac{t}{\mu^{p}} -  \frac{1}{\mu ^{p}} \sum_{j=1}^{p-2} \frac{x^{j}\mu ^{j+1}}{ j!} - T_{p}(x) 
\end{equation*}
so if $\mu^{j} \geq 0$ for every odd number $j \geq 3$, $j \neq p$, then $h^{(1)}_{p}$ is strictly decreasing on $(0,\infty)$. Hence, there is a unique positive solution for the equation   $h^{(1)}_{p}(x,t,\mu ^{2},\ldots,\mu^{p} ) = 0$ which implies that the set $G_{p-2}(\alpha_{0},\ldots,\alpha_{p-2})$ consists only one element (see the proof of part (i)). 

  (iii) Assume that $p=2$. Then the unique solution to the equation $\alpha_{0} = \exp(x)$ is $\ln(\alpha_{0})$. Thus,   $G_{2}(\alpha_{0}) = \{ y \} $ where $$y=\ln \left (1 + \frac{tb}{\mu^{2}} \right ).$$ Plugging $y$ into inequality (\ref{Ineq: Ben useful}) proves part (iii). 
  
  (iv) Assume that $p=3$. From part (ii) $G_{3}$ consists of one element. Note that $bx^{2} \geq \max (x^{3},0)$ for all $x \leq b$. Thus, $b\mu^{2}_{i} \geq \mu^{3}_{i}$ for all $i=1,\ldots,n$. Hence, $\alpha _{1}$ is non-negative. Because $\alpha _{0} >1$ and $\alpha _{1} >0$ (if $\alpha_{1}=0$ we get Bennett's inequality as in part (iii)),  $G_{3}(\alpha_{0},\alpha_{1}) = \{ y \} $ where $y$ is the unique and  positive solution to the equation $\exp (x) = \alpha _{0} - \alpha_{1}x$ that is given by $$y=\frac{\alpha _{0}}{\alpha_{1}} - W\left ( \frac{\exp (\alpha_{0} / \alpha_{1}) }{\alpha_{1}} \right )$$
  where $W$ is the Lambert $W$-function (see \cite{corless1996lambertw}). Plugging $y$ into inequality (\ref{Ineq: Ben useful}) proves part (iv).
\end{proof}

\bibliographystyle{ecta}
\bibliography{Concent}

\end{document}